\documentclass[12pt]{amsart}
\usepackage{amsmath,latexsym}
\usepackage{graphicx}

\newcommand{\R}{I\!\!R}
\newtheorem{theorem}{Theorem}[section]
\newtheorem{lemma}{Lemma}[section]
\newtheorem{proposition}{Proposition}[section]

\newtheorem{corollary}{Corollary}[section]
\newtheorem{remark}{Remark}[section]

\title{ Lack of contact in a lubricated system }

\author{Ionel Sorin Ciuperca$^1$ and Jos\'e Ignacio Tello$^2$}

\begin{document}

\maketitle

\begin{flushleft}

$^1$ Universit\'e de Lyon, CNRS, Universit\'e Lyon 1, Institut Camille
Jordan UMR5208, B\^at Braconnier, 43 Boulevard du 11 Novembre 1918,
F-69622, Villeurbanne, France.

\vspace{2mm}

$^2$  Matem\'atica Aplicada. E.U.I. Inform\'atica. Universidad Polit\'ecnica de Madrid. 28031 Madrid. Spain

\end{flushleft} 

\begin{abstract}

We consider the  problem of a rigid surface moving over a flat plane.
The surfaces are separated by a small gap
 filled by a lubricant fluid. The relative position of the 
surfaces is unknown except for the
initial time $t=0$. The total load applied over the upper surface  is a know constant for $t>0$.
The mathematical model consists in a
coupled system formed by Reynolds variational inequality for 
incompressible fluids and Newton$^{\prime}$s second Law.
 In this paper we study  the global existence  and uniqueness
 of solutions of the evolution problem when the position of the surface
 presents only one  degree of freedom, under extra assumptions on its geometry.
The existence of steady states is also studied.
\end{abstract}

{\small \underline{Key words}: lubricated systems, 
Reynolds variational inequality,  
global solutions, sationnary solutions.

\underline{A.M.S. subject classification}. 35J20, 47H11, 49J10. }

\section{Introduction}

Lubricated contacts are widely used in mechanical systems
to connect solid
bodies that are in relative motion. A lubricant fluid is introduced
in the narrow space between the bodies with the purpose of avoiding
direct solid-to-solid contact.
\\
This contact is said to be in the hydrodynamic regime, and the
forces transmitted between the bodies result from the shear and
pressure forces developed in the lubricant film.

We consider one of the
simplest lubricated systems which consists of two rigid surfaces   in
hydrodynamic contact. The bottom surface, assumed planar and
horizontal moves with a constant horizontal translation velocity
and a vertical given force $ F > 0 $ is applied
vertically on the upper body.

The wedge between the two surfaces is filled with an incompressible
fluid. We suppose that the wedge satisfy the thin-film hypothesis,
so that a Reynolds-type model can be used to  describe the problem.

We denote by $\Omega$ the two-dimensional domain in which the
hydrodynamic contact occurs. We assume
 that $\Omega$ is open,
bounded and with regular boundary $\partial \Omega$.
Without lost of generality we consider  $ 0 \in \Omega$.
We assume that the upper body, the {\it slider},
 is allowed to move only by vertical translation. The normalized
distance between the surfaces is given by
$$
h(x,t) = h_0(x) + \eta(t)
  $$
where $ \eta(t) > 0 $ represents the vertical translation of the
slider and
\\
$h_0: \Omega \longrightarrow [0, \infty[$ describes the shape
of the slider and
is a given function
satisfying
\begin{equation}
\label{hyp-base-h0}
h_0 \in C^1({\bar \Omega}), \quad
\min_{x \in \Omega} h_0(x)=h_0(0)=0.
\end{equation}
%
The mathematical model we study considers the possible cavitation in the
thin film, so
the (normalized) pressure ``$p$'' of the fluid satisfies the Reynolds
variational inequality (see \cite{frene}):
 \begin{equation}
\int_\Omega h^3 \nabla p \cdot \nabla (\varphi-p) \geq \int_\Omega h
\dfrac{\partial }{\partial x_1} (\varphi-p) -
\eta^{\prime}(t)\int_\Omega (\varphi-p), \ \  \forall \varphi \in
K \label{2}
\end{equation}
where
$$
K =  \left\{ \varphi \in H^1_0(\Omega): \varphi \geq 0 \right\},
$$
and ``$\nabla$'' denotes the gradient
with respect to the variables $ x \in \Omega$.
 Without lost of generality we assume the velocity of the bottom surface is oriented in the
direction of the $x_1$ - axis and its normalized value is equal to 1.
\\
The equation of motion of the slider is
\begin{equation}
\eta^{\prime \prime} =\int_{\Omega} pdx -F \qquad \mbox{ (second
Newton Law)} \label{3}
\end{equation}
completed with the initial conditions:
\begin{equation}
\eta(0) =\eta_0 \label{4}
\end{equation}
\begin{equation}
\eta^{\prime }(0) =\eta_1,\label{5}
\end{equation}
where  $\eta_0>0$,  $\eta_1\in \R$ are given data.
\newline
The unknowns of the problem are the pressure $p(x,t)$ and the
vertical displacement of the slider  $\eta(t)$.
It is known that for any given $C^1$ function $\eta(t)$ the problem
\eqref{2} is well posed (see for instance \cite{kinder}).
The system (\ref{2})-(\ref{5}) is equivalent to the
following Cauchy problem for a second order ordinary differential
equation in $\eta$:
\begin{equation}
\left\{
\begin{array}{l}
\eta^{\prime \prime}=G(\eta, \eta^{\prime})
\\
 \eta^{\prime }(0)
=\eta_1,
\\ \eta(0) =\eta_0,
\end{array}
\right. \label{6}
\end{equation}
where $G:]0, \infty[ \times \R \longrightarrow  \R$ is given by
$$G(\beta, \gamma):=\int_{\Omega}q(x)dx -F,$$
and $q\in K$ (depending on $\beta$ and  $\gamma$) is the unique
solution to
\begin{equation}
\left\{
\begin{array}{l} \displaystyle
\int_\Omega\left (h_0 + \beta \right )^3 \nabla q \cdot \nabla
(\varphi-q) \geq \int_\Omega h_0
\dfrac{\partial }{\partial x_1} (\varphi-q) - \gamma
\int_\Omega (\varphi-q) \\ \\  \forall \varphi \in K.
\end{array}
\right. \label{7}
\end{equation}
The main goal of the paper is to give sufficient conditions on the shape $h_0$
of the slider to obtain global existence on time to \eqref{6}, i.e.  there is no contact solid-to-solid for $t< \infty$.
We also study the existence of steady states  of the problem.
Another interesting physical question
which we adress here  is to see if there exists a
``barrier'' value $\eta_b > 0$ such that
$ \eta(t) \geq \eta_b, \ \ \forall \ t > 0 $.
\\
We prove the existence of $\eta_b$ for two of the three cases studied. Third case (the so called ``flat case''),
we prove that $\eta$ tends to 0 as $t \rightarrow \infty.$
\\
The main ideas of these results are the following:   when the distance between the surfaces decreases
(i.e. $\eta'\leq 0$)  there exists a lower bound of the force exerted by the
pressure of the fluid on the upper body. This lower bound
admits an expression of the form
$F_S + F_D$, where $F_S$ is a ``spring-like'' force
and $F_D$ is a ``damping force'' (see Corollary (\ref{cor1})
and  Remark \ref{rem3.1}).

$F_S$ depends only on the position $\eta(t)$ and
represents the force exerted by the pressure of the fluid
for the stationary position in an auxiliary  sub-domain $U$ of $\Omega$.

 $F_D$ is of the form $ F_D = - \eta' d $ where $d$ is
a ``dumping'' coefficient and depends only on $\eta$.
The global   existence of the solution $\eta$ is a consequence
of the  velocity of blow up of $d$
 when $\eta$ tends to 0.
The
 existence of a ``barrier'' $\eta_b$ is based on the fact
that $F_S$  blows up when $\eta$ tends to 0.
In the ``flat case'' the force $F_S$ is equal to zero, which
explains the non existence of a barrier.

The present work is related to different articles on the fluid-rigid
interaction problems (see for example \cite{conca}, \cite{desjardins},
\cite{ger-var-hill}, \cite{hesla} and \cite{hill-07}, for a
non-exhaustive bibliography on this subject).
These papers concern the study of the motion of one or many rigid
bodies inside a domain
$ Q \in \R^n, n = 2, 3 $, filled with an
incompressible fluid with constant viscosity. The mathematical
model is a coupled system between Navier-Stokes equations modeling  the fluid
and second Newton Law to describe the rigid bodies positions.
A relevant problem in this context is the so called
``non-collision'' problem, where the question is to know if this body will touch the
boundary $\partial Q$ of the fluid in finite time.
\\
In \cite{hill-07} Hillairet consider the particular case where $Q$ is the
half-plane $ \R \times \R_+ $ and the rigid body is a disk
which moves
only along the vertical axis.
He proves that in absence of external forces the solution is defined globally in time. He also shows that the disk
remains all the time ``far'' from the boundary.
\\
In \cite{ger-var-hill} G\'erard-Varet  and Hillairet consider a more general
shape of the rigid body in a general domain $Q$ in presence of gravity.
 They prove the existence of a
global in time solution of the problem, but now
the rigid body can go the boundary of the domain as $t$ goes to infinity.
Similar results are given by Hesla in \cite{hesla}.

The main difference between the above mentioned
works and the present one is
the obtention  in this study of a ``barrier'' value
$\eta_b > 0$ for any exterior force $F$. We can explain this
difference by the high shear and pressure that develop in a lubricant
fluid film, due especially to the relative motion of the closed
surfaces.
An interesting open question is to see if similar ``barrier'' results
can be obtained
for situations
when the thin film hypothesis is not satisfied in the fluid
(so the  full Navier-Stokes equations must be  used in the place of
Reynolds models),  but relative horizontal motion exists between the
two surfaces.

Fluid-rigid interaction problems in lubrication where also considered
in \cite{diaztello} where Reynolds equation is used in the place
of Reynolds variational inequality in the particular
``flat'' case. We also mention the papers  \cite{busciuhafjai07a},
\cite{busciuhafjai07b} and
\cite{ciujaitel1}, where
the existence of steady states is studied for lubricated
systems with two degrees of freedom.

The contents of the paper are the following:
\\
In Section \ref{s2} we  precise the hypothesis on $h_0$
and present the main results of the paper. In Section \ref{s3} we give
some preliminary results and Section \ref{s4} is devoted to the proof
of the theorems of Section \ref{s2}.

\section{Main results}
\label{s2}
We begin by the local in time existence and uniqueness result, for
which the minimal hypothesis \eqref{hyp-base-h0} is sufficient.
\begin{theorem}
\label{G-lipsch}
The function $G$ is locally Lipschitzian, so
we have the existence and uniqueness of solution to (\ref{6})  locally in time.
\label{t1}
\end{theorem}
Let $[0,T[$ be the maximal interval of existence of solution to \eqref{6},
 so $\eta \in C^2([0,T[)$.
\\
The main goal of the paper is to prove  that $ T = + \infty $.
\newline
It is equivalent to prove that for any  fixed $T>0$ there
exists $m>0$ and $M>0$ (depending eventually on $T$) such that
\begin{equation}
\left\{
\begin{array}{ll}
m \leq \eta (t) \leq M, &  \mbox{ for all } t \in [0,T[  \\
 |\eta ^{\prime} (t) | \leq M. &
\end{array}
\right. \label{8}
\end{equation}
Moreover, we are interested to know  if there exists such  constants $m$
and $M$ independent on  $T$.
%

In order to study the existence of steady states and global existence of
solutions to \eqref{6}
we consider three different cases depending on the shape of the
slide $h_0$.
\newline
{\bf Case I. Line contact} \newline
We assume that $h$ is equal to $0$ only in the line $\{x_1=0\}$
i.e.
$$
 \left\{
\begin{array}{l}
h_0(0,x_2)=0   \ \mbox{ for all } x_2 \in \R \mbox{ such that } (0,x_2) \in \Omega \mbox{ and }  \\
h_0(x_1,x_2)>0 \ \mbox{ for all } (x_1, x_2) \in  \Omega, \  x_1
\neq 0.
\end{array}
\right.
  $$
We also assume that there exists $\alpha \geq 1 $ such that
\begin{equation}
h_0(x_1,x_2) \sim |x_1|^{\alpha}  \mbox{ when } x_1 \rightarrow 0.
\label{9}
\end{equation}
More precisely there exists a neighborhood $W$ of $0$ and a
function $h_1$ regular enough on the closure $ \bar{W}$
of $W$ with
$ h_1 > 0 $ on $ \bar{W}$ such that
$$
h_0(x_1,x_2) = |x_1|^{\alpha} h_1(x_1,x_2) \quad \text {in} \quad  W.
  $$
\newline
{\bf Case II. Point contact}
\newline
We assume that $h_0$ is equal to $0$ only in the point $\{x=0\},$
i.e.
$$h_0(0)=0 \mbox{ and } h_0(x)>0 \mbox{ for all } x \in \Omega -
\{0\}.$$
We also assume that there exists $ \alpha \geq 1 $ such that
\begin{equation} h_0(x) \sim
|x|^{\alpha} \mbox{  when } x \rightarrow 0 \label{10}
\end{equation}
that is, there exist $W$ and $h_1$ as in {\bf Case I} such that
$$
h_0(x)=|x|^{\alpha} h_1(x) \quad \text{in} \quad  W
  $$
(where $|\cdot|$ is the
euclidian norm in $\R^2$).
\newline
{\bf Case III. Flat slides}
\newline
We  assume that  $h$ is flat, i.e.
\begin{equation}
\label{h0e0}
h_0=0 \quad \text{ on } \; \Omega
\end{equation}
which implies
$h(x,t)=\eta(t)$.
\\

The results concerning the existence of steady states for cases I and II are enclosed in the following theorem:
\begin{theorem}
\label{th-ex-station}
Let $h_0$ satisfy assumption \ref{9} in case I  or \ref{10} in case II for $\alpha$ satisfying
\begin{equation}
\begin{cases}
\alpha > 1 & \quad \text{ in Case I (line contact) }
\\
\alpha > \frac 3 2 & \quad \text{ in Case II (point contact) }.
\end{cases}
\end{equation}
Then there exists at least one stationary solution
$\bar{\eta} > 0 $ of the Cauchy problem \eqref{6}, i.e.
$$
G(\bar{\eta}, 0) = 0.
   $$
\end{theorem}
\begin{remark}
The problem of uniqueness of the stationary solution is a difficult one.
In \cite{busciuhafjai07b} the authors proved the uniqueness of solutions to the
1-dimension problem  for a particular  function $h_0$.
\end{remark}
Results of global existence and barrier functions are presented  in the following theorem:
\begin{theorem}
\label{th-ev}
We assume that $h_0$
 satisfy assumption \ref{9} in case I  or \ref{10} in case II for $\alpha$ satisfying
\begin{equation}
\begin{cases}
\alpha \geq  \frac 3 2 & \quad \text{ in Case I (line contact) }
\\
\alpha \geq  2 & \quad \text{ in Case II (point contact) },
\end{cases}
\end{equation}
then $ T = + \infty $.
Moreover, there exist constants $m_0, M_0$ and $M_1$ such that
$ 0 < m_0 \leq  M_0 $ and $M_1 \geq 0 $ satisfying
$ \forall t \geq 0 $:
\begin{equation}
\nonumber
\begin{cases}
m_0 \leq \eta(t) \leq M_0
\\
|\eta'(t) | \leq M_1,
\end{cases}
\end{equation}
for $t\geq 0$.
\end{theorem}
\begin{remark}
For the one-dimensional problem, i.e.
$\Omega$ is an interval of $\R$,
the results are the same than in the case II (line contact) for the two-dimensional
problem.
\end{remark}
Some relevant questions  concerning the dynamical
system \eqref{6} remain open:
\\
- Uniqueness of solution for the steady states.
\\
- Stability of the steady states.
\\
- Existence of periodic solutions.
\\
- The Attractor of the dynamical system.
\begin{theorem}
\label{th-ev-fc}
We assume that $h_0 \equiv 0$ (Case III), then $ T = + \infty $,
moreover there exist $ M_0, M_1 > 0 $ such that
\begin{equation}
\nonumber
\begin{cases}
0 < \eta(t) \leq M_0
\\
|\eta'(t) | \leq M_1,
\end{cases}
\end{equation}
for $  t \geq 0 $ and
$$
\eta(t) \rightarrow 0 \quad \text{ as } \quad t \rightarrow + \infty.
  $$
Moreover
there exist $t_0 \geq 0, \;  a > 0 $
and $b \in \R$ with $t_0 + b > 0$, such that
$$
\eta(t) \geq \frac a {\sqrt{t + b}} \quad  \quad \forall \, t \geq t_0.
  $$
In addition,  no stationary solution
exist for the system \eqref{6}.
\end{theorem}
\begin{remark}
The same result can be obtained for the corresponding one-dimensional problem.
\end{remark}

\section{Some preliminary results on the function $G$}
\label{s3}
\subsection{ Results for $h_0$ satisfying \eqref{hyp-base-h0} (all cases are included).}
In this subsection we proof some preliminary results on $G$ under the minimal hypothesis
\eqref{hyp-base-h0}.
Let $V_1$ be defined as follows
\begin{equation}
\label{V1}
V_1 = \sup_{x \in \Omega} \left\{ - \frac {\partial h_0} {\partial x_1}
(x) \right\}
\end{equation}
It is clear that $ V_1 \geq 0 $.
\begin{lemma}
\label{ineq-G-1}
$i)$ \quad There exists $c_1 > 0$ such that
$$
G(\beta, \gamma) \leq \frac {c_1} {\beta^3} - F \quad \quad
\forall \beta > 0, \; \gamma \geq 0.
  $$
$$
ii) \quad \quad \quad  G(\beta, \gamma) = -F
\quad \quad \forall \beta > 0, \;  \gamma \geq V_1.
  $$
\end{lemma}
\begin{proof}
$i)$ \ We take  $\varphi = 0$ and $\varphi = 2 q$ in
\eqref{7} to have
$$
\int_\Omega (h_0 + \beta)^3 | \nabla q |^2 =
\int_\Omega h_0 \frac {\partial q} {\partial x_1} - \gamma
\int_\Omega q.
  $$
We use the inequalities $ h_0 + \beta \geq \beta, \; \gamma
\geq 0, \; q \geq 0 $  to  obtain
$$
\beta^3 \int_\Omega | \nabla q |^2 \leq
\int_\Omega h_0 \frac {\partial q} {\partial x_1}.
  $$
We use the  Poincar\'e inequality and the proof of case i) ends.
\\
$ii)$ \
The inequality \eqref{7} can be written:
$$
\int_\Omega (h_0 + \beta)^3 \nabla q \cdot \nabla
(\varphi - q)  \geq -
\int_\Omega \left ( \frac {\partial h_0} {\partial x_1} + \gamma
\right ) (\varphi - q) \, dx, \quad \forall \varphi \in K.
  $$
Since
$$
\gamma + \frac {\partial h_0} {\partial x_1} \geq 0, \quad
\forall \; x \in \Omega
  $$
we have that  $ q = 0 $ which gives the result.
\end{proof}
The rest of the results enclosed in this section concern the function
$G(\beta, \gamma)$ when $\gamma \leq 0$.
We begin by a general result on variational inequalities.
\begin{lemma} \label{l1}
Let $a\in L^{\infty}(\Omega)$ such that  $\inf_\Omega a >0$. Let   $f \in H^{-1}(\Omega)$
and $q\in K$ be the solution of the problem
\begin{equation}
\int_{\Omega} a\nabla q \cdot \nabla (\varphi-q) \geq
\; <f, \varphi-q>,
\qquad \mbox{for all } \; \varphi \in K.  \label{11}
\end{equation}
Let $U \subset \Omega $ arbitrary and open, and let $r \in
H^{1}_0(U)$
the solution to
\begin{equation}
\int_{U} a\nabla r \cdot \nabla \psi = \; <f, \psi>, \qquad \mbox{for
all } \; \psi \in H_0^1(U).  \label{12}
\end{equation}
Then $q \geq r$ on $U$.
\end{lemma}
\begin{proof} We consider $\psi\in H^{1}_0(U),$ $\psi \geq 0 $ arbitrary,
and we extend it to $\Omega$ by $0$ and denote the extended function by $\tilde{\psi}$ which belongs to $K$.
For simplicity we omit the tilde.
We take  $\varphi=q+
\psi$ in (\ref{11}) to obtain
\begin{equation}
\int_{U}a \nabla q \cdot \nabla \psi  \geq \; <f, \psi>, \label{13}
\end{equation}
Let us denote $\xi = q-r$. From (\ref{12}) and (\ref{13}) we have
\begin{equation}
\int_{U}a \nabla \xi \cdot \nabla \psi  \geq 0, \label{14}
\end{equation}
 for any $\psi \in H^1_0(U)$, $\psi\geq 0$. On the other hand
we have $\xi \geq 0$ on $\partial U$.
From the maximum principle we obtain $ \xi \geq 0 $ on $U$
which proves the lemma.
\end{proof}
The following result is a consequence of the above lemma and
non-negativity of
the solution $q$ to \eqref{7} in $\Omega$.
\begin{corollary}
\label{cor1}
Let us denote for any open set $U \subset \Omega$ and any $\beta >0$
by $q_{1 \beta}$ and $q_{2 \beta}$ the solutions to the following
problems
\begin{equation}
\left\{
\begin{array}{ll}
- \nabla \cdot [(h_0+ \beta)^3 \nabla q_{1 \beta}]=- \frac{\partial
h_0}{\partial x_1} & \mbox{ on } U \\
& \\
q_{1 \beta}=0, & \mbox{ on } \partial U
\end{array}
\right. \label{15}
\end{equation}
and
\begin{equation}
\left\{
\begin{array}{ll}
- \nabla \cdot [(h_0+ \beta)^3 \nabla q_{2 \beta}]=1 & \mbox{ on } U \\
& \\
q_{2 \beta}=0, & \mbox{ on } \partial U
\end{array}
\right. \label{16}
\end{equation}
respectively.
We then have
\begin{equation}
G(\beta, \gamma)\geq \int_{U}q_{1 \beta}dx- \gamma\int_{U}q_{2
\beta}dx-F. \label{17}
\end{equation}
for all $\beta>0$, $\gamma \in \R$ and  $U \subset \Omega $ open.
\end{corollary}
\begin{remark}
The expressions $\int_{U}q_{1 \beta}dx$ and
$\int_{U}q_{2 \beta}dx$ represent the force
``$F_S$'' and  the damping  coefficient ``$d$'' respectively,  as we described
in the Introduction.
\label{rem3.1}
\end{remark}
\subsection{The case of non-horizontal slider.}
In this subsection we assume $h_0 \neq 0$ and also that $h_0$ satisfies the hypothesis of Cases I or II
(line contact and point contact case respectively).
We prove the existence of a sub-domain
$U \subset \Omega$ such that the averages of the corresponding functions
$q_{1 \beta}$ and $q_{2 \beta}$ are
``large'' in some sense when $\beta$ is small.
\\
We denote by $\rho \geq  0$ and $\theta \in [0, 2 \pi]$
the polar coordinates of $(x_1, x_2)$.
\begin{lemma} \label{l2}
\begin{itemize}
\item[a)] Case I. (Line contact) \newline
There exist $\delta, \beta_0,c_2>0$ and $B_{l, \beta}$ defined by
$$B_{l, \beta}:=]- 2 \beta^{1/\alpha}, -\beta^{1/\alpha}[
\; \times \;
]-\delta,\delta[
  $$
such that for any $0< \beta \leq \beta_0$ we have
$$\frac{\partial h_0}{\partial x_1} \leq
-c_2\beta^{1- 1/\alpha} \qquad \mbox{ on } B_{l, \beta}.
  $$
\item[b)] Case II. (Contact point) \newline
There exists $\theta_0 \in \, ]0, \frac{\pi}{2}[$,
 $c_2$, $\beta_0>0$ and the sector $B_{p, \beta}$ defined by
$$B_{p, \beta}= \{(x_1,x_2) \in \R^2; \ \beta^{1/\alpha}
\leq \rho \leq 2 \beta^{1/\alpha}; \ \pi - \theta_0\leq \theta \leq \pi +
\theta_0 \}
   $$
such that for any $0< \beta \leq \beta_0$:
$$\frac{\partial h_0}{\partial x_1} \leq -c_2\beta^{1-
1/\alpha} \quad \mbox{ on } B_{p, \beta}.
  $$
\end{itemize}
\end{lemma}
\begin{proof}
\begin{itemize}
\item[a)] We have for $x_1\leq 0$
$$\frac{\partial h_0}{\partial x_1}=- \alpha (-
x_1)^{\alpha-1}h_1+(-x_1)^{\alpha} \frac{\partial h_1}{\partial x_1}
= (-x_1)^{\alpha-1}h_1(x)\left[- \alpha-x_1\frac{\frac{\partial
h_1}{\partial x_1}}{h_1(x)} \right].
$$
Since $h_1>0$ on ${\bar W}$ we obtain  $$x_1 \frac{\frac{\partial
h_1}{\partial x_1}}{h_1(x)} \longrightarrow 0 \quad \mbox{ when } x
\longrightarrow 0$$ and the result is obvious.
\item[b)] For any $x$ in $W-\{0\}$ we have
$$\frac{\partial h_0}{\partial x_1}=\alpha |x|^{\alpha-1}
\frac{x_1}{|x|}h_1+|x|^{\alpha}\frac{\partial h_1}{\partial
x_1}= |x|^{\alpha-1}h_1\left(\alpha \frac{x_1}{|x|}+|x|\frac{\frac{\partial
h_1}{\partial x_1}}{h_1}\right).
$$
Now we can chose $\theta_0 \in ]0, \frac{\pi}{2}[$ such that $
\frac{x_1}{|x|}<-\frac{1}{2}$ if $\pi- \theta_0 \leq \theta \leq
\pi +\theta_0$ (choose for example $\theta_0=\frac{\pi}{6}$). On the
other hand we have $$|x|\frac{\frac{\partial h_1}{\partial
x_1}}{h_1} \longrightarrow 0 \qquad \mbox{ when } \quad x
\longrightarrow 0$$ which proves the lemma.
\end{itemize}
\end{proof}
\begin{lemma} \label{l3}
Let us consider $q_{1 \beta}, \; q_{2 \beta}$ the solutions to   (\ref{15})-(\ref{16})
where $U$ is given by
$$
U:=
B_{l, \beta} \quad \mbox{ in case I,}
   $$
$$
U:= B_{p, \beta} \quad \mbox{ in case II,}
  $$
with $B_{l, \beta}$ and $B_{p, \beta}$ defined in Lemma \ref{l2}.
Then there exists $\beta_0, \; c_3, \; c_4 >0$ such that for
any $\beta \in ]0,  \beta_0]$ we obtain
\begin{equation}
\left\{
\begin{array}{ll}
\int_{B_{l,  \beta}}q_{1 \beta}(x)dx \geq c_3
\beta^{2\left(1/\alpha - 1\right)} & \mbox{ in case I}
\\
\int_{B_{p,  \beta}}q_{1 \beta}(x)dx \geq c_3
\beta^{3/\alpha - 2} & \mbox{ in case II}
\end{array}
\right.
\label{19}
\end{equation}
 moreover
\begin{equation}
\left\{
\begin{array}{ll}
\int_{B_{l,  \beta}}q_{2 \beta}(x)dx \geq c_4
\beta^{3\left(1/\alpha - 1\right)} & \mbox{ in case I}
\\
\int_{B_{p,  \beta}}q_{2 \beta}(x)dx \geq c_4
\beta^{4/\alpha - 3} & \mbox{ in case II}.
\end{array}
\right.
\label{18}
\end{equation}
\end{lemma}
\begin{proof} From (\ref{16}) we deduce
\begin{equation}
\int_{U} q_{2 \beta}dx =\int_{U}(h_0+ \beta)^3|\nabla q_{2
\beta}|^2dx \label{20}.
\end{equation}
From the equality $$\int_{U}(h_0+ \beta)^3\nabla q_{2 \beta}
\cdot \nabla \varphi =\int_U \varphi, \qquad \mbox{ for all } \varphi \in
H^{1}_0(U)$$ and Cauchy-Schwarz inequality, we get
\begin{equation}
\int_{U}(h_0+ \beta)^3|\nabla q_{2 \beta}|^2 dx \geq \sup_{\varphi\in
H^1_0(U), \ \varphi \neq 0 } \left[ \frac{ \left(\int_U \varphi
dx \right)^2}{\int_U(h_0+\beta)^3 |\nabla \varphi|^2} \right].
 \label{21}
\end{equation}
It  suffices to find appropriate test functions $\varphi \in
H^1_0(U)$, $\varphi \neq 0$ such that the term
$$\frac{ ( \int_U\varphi
dx)^2 }{\int_U(h_0+\beta)^3 |\nabla \varphi|^2} $$ is large enough.

{\bf Proof of (\ref{18})}
\\
{\bf Case I: Line contact.}
 We choose $$\varphi(x_1, x_2)=\psi_1\left( \frac{x_1}{\beta^{1/\alpha}} \right)
 \psi_2(x_2)$$
with $\psi_1 \in {\mathcal D}( ]-2, -1[ )$,
$\psi_1 \geq 0$, $\psi_1 \not \equiv 0$ and
$ \psi_2 \in D ( ]- \delta_2, \delta_2[ ), \; \psi_2 \geq 0, \;
\psi_2 \not \equiv 0 $.
Then
$$\int_{B_{l, \beta}} \varphi \, dx =
\int_{-2\beta^{1/\alpha}}^{-\beta^{1/\alpha}}\int_{-
\delta_2}^{\delta_2} \psi_1\left(\frac{x_1}{\beta^{1/\alpha}}\right)
\psi_2(x_2)dx_1 dx_2=$$
$$ \beta^{1/\alpha} \int_{-2}^{-1}
\psi_1(y_1)dy_1\int_{- \delta_2}^{\delta_2} \psi_2(x_2)dx_2
  $$
and
$$
\int_{B_{l, \beta}} (h_0+\beta)^3 |\nabla \varphi|^2dx =$$
$$
\int_{-2\beta^{1/\alpha}}^{-\beta^{1/\alpha}}\int_{-
\delta_2}^{\delta_2} (h_0+ \beta)^3
\left[
\frac{1}{\beta^{2/\alpha}}|\psi^{\prime}_1\left(\frac{x_1}
{\beta^{1/\alpha}}\right)
\psi_2(x_2)|^2+ |\psi_1\left(\frac{x_1}{\beta^{1/\alpha}}\right)
\psi_2^{\prime}(x_2)|^2dx
\right].
$$
It is easy to show that
$$\int_{B_{l, \beta}}(h_0+ \beta)^3|\nabla \varphi|^2dx \leq
c_2'\beta^{3- 1/\alpha}
  $$
where $c_2'>0$ is  a constant independent on $\beta$.
 (\ref{21}) implies
$$\int_{B_{l, \beta}}(h_0+\beta)^3|\nabla q_{2 \beta}|^2dx \geq
c_3 \beta^{3(1/\alpha - 1)}
  $$
and thanks to (\ref{20}) we obtain (\ref{18})$_1$.
\\
{\bf Case II: Contact point.}
We choose  $\varphi(x_1, x_2)= \psi_3(\frac{\rho}
{\beta^{1/\alpha}}) \psi_4(\theta)$ with $\psi_3 \in
{\mathcal D}(]1, 2[)$,
$\psi_3 \geq 0$, $\psi_3 \not \equiv 0$,
and
$\psi_{4} \in {\mathcal D}(]\pi- \theta_0, \pi+ \theta_0[),
$ $ \psi_4 \geq 0$, $\psi_4 \not \equiv 0$. In
polar coordinates, we have
$$
\int_{B_{p, \beta}}\varphi \, dx=
\int_{\beta^{1/\alpha}}^{2\beta^{1/\alpha}} \int_{\pi - \theta_0}^
{\pi + \theta_0}
\psi_3\left(\frac{\rho}{\beta^{1/\alpha}}\right)\psi_4(\theta)
\rho d \rho d\theta=$$
$$
\beta^{2/\alpha}
\int_1^{2}\psi_3(\rho_1)\rho_1 d \rho_1 \int_{\pi - \theta_0}^{\pi +
  \theta_0} \psi_4(\theta) \, d \theta
$$
and
$$
\int_{B_{p, \beta}}(h_0+ \beta)^3|\nabla \varphi|^2dx=$$
$$
\int_{ B_{p, \beta}} (h_0+ \beta)^3\left[
\frac{1}{\beta^{2/\alpha}}\left| \psi_3^{\prime}
\left(\frac{\rho}{\beta^{1/\alpha}}\right)   \right|^2
|\psi_4(\theta)|^2 + \frac 1 {\rho^2} \left| \psi_3
\left(\frac{\rho}{\beta^{1/\alpha}}\right)   \right|^2
|\psi^{\prime}_4(\theta)|^2
\right].
 $$
We easily obtain
%
$$
\int_{B_{p, \beta}}(h_0+\beta)^3|\nabla \varphi|^2 dx \leq c_2'\beta^3,$$
where $c_2'>0$ is a constant independent of $\beta$.
From (\ref{21}) we obtain
$$\int_{B_{p, \beta}}(h_0+ \beta)^3|\nabla q_{2 \beta}|^2dx \geq c_3
\beta^{4/\alpha-3},
  $$
and by (\ref{20}) we get (\ref{18})$_2$.
\\
{\bf Proof of (\ref{19})}
\\
From lemma \ref{l2} we have
$$- \frac{\partial h_0}{\partial x_1} \geq c_2 \beta^{1- 1/\alpha} \qquad
\mbox{ on }
B_{l, \beta} \qquad \mbox{(in case I)}$$
and
$$- \frac{\partial h_0}{\partial x_1} \geq c_2
\beta^{1- 1/\alpha} \qquad
\mbox{ on }
B_{p, \beta} \qquad \mbox{(in case II)}.$$
By maximum principle we deduce the inequality
$$q_{1\beta} \geq c_2 \beta^{1- 1/\alpha}q_{2 \beta},
 \qquad \mbox{ on } B_{l, \beta} \qquad (\mbox{respectively  } B_{p, \beta}).$$
By (\ref{18}) the proof ends.
\end{proof}
The following corollary is a consequence of Corollary \ref{cor1} and Lemma \ref{l3}.
\begin{corollary}
For any $ \beta \in \; ]0, \beta_0] $ and
$ \gamma \leq 0 $ we have
$$
G(\beta, \gamma) \geq c_3 \beta^{2(1/\alpha -1)}-\gamma
c_4 \beta^{3(1/\alpha-1)}-F,
\quad \mbox{ in case I}$$
or
$$G(\beta, \gamma) \geq c_3 \beta^{3/\alpha -2} - \gamma
c_4 \beta^{4/\alpha -3}-F,
\quad \mbox{ in case II}.$$
with $\beta_0, c_3, c_4$ as in Lemma \ref{l3}.
\label{c2}
\end{corollary}
\section{Proof of the main results}
\label{s4}
We consider $\eta(t)$ the solution of the Cauchy problem \eqref{6} defined on the maximal interval
$[0, T[$.
\subsection{Bounds  on $\eta$
for the non-horizontal slider case.}
In this subsection we assume  that $h_0$ satisfies the hypothesis of Cases I or II
(line contact and point contact case respectively).
We prove that $\eta$  and $\eta^{\prime}$ are bounded and $\eta$ reminds ``far" from 0.
We first prove in Proposition \ref{m-1} that
$\eta'$ admits an upper bound and the same for $\eta$ in Proposition \ref{m-2}.
These results are needed  to prove the existence of lower bounds for $\eta^{\prime}$ (Propostion \ref{m-3})
and $\eta$ (Proposition \ref{m-4})
\\
Let $V_2$ be defined by
\begin{equation}
V_2 = \max \{\eta_1 + 1, V_1\}
\end{equation}
for $V_1$ as in  \eqref{V1}, then we have:
\begin{proposition}
\label{m-1}
$$
\eta'(t) < V_2 \quad \quad \forall t \in [0, T[.
  $$
\end{proposition}
\begin{proof}
We argue by the contrary and assume that $t_1 > 0 $ is the first point  such that
$
\eta'(t_1) = V_2,
  $ \
which implies
$
\eta''(t_1) \geq 0
  $
which contradicts Lemma \ref{ineq-G-1}  $ii)$ where
$
\eta''(t_1)  = -F.
   $
\end{proof}
 We introduce two energies
$ E_1, E_2 : ]0, + \infty [ \times \R \rightarrow \R  $  defined  by
%
$$
E_1(\beta, \gamma)   = \frac 1 2 \gamma^2 + F \beta
   $$
and
$$
E_2(\beta, \gamma)   =  \frac 1 2 \gamma^2 + F \beta + \frac {c_1} {2 \beta^2}
   $$
%
for $c_1$ as in Lemma \ref{ineq-G-1}.
\\
The energies $E_1$ and $E_2$ are used in the  following lemma when  $\eta(t)$ is
non-increasing or
non-decreasing respectively.
\begin{lemma}
\label{l3-1}
For any $ t \in [0, T[ $ we have
$$
i) \quad \quad \frac d {dt} E_1(\eta(t), \eta'(t)) \leq 0 \quad \text{ if } \quad \eta'(t) \leq 0
   $$
$$
ii) \quad \quad \frac d {dt} E_2(\eta(t), \eta'(t)) \geq 0 \quad \text{ if } \quad \eta'(t) \geq 0.
   $$
\end{lemma}
\begin{proof}
$i)$ \ We multiplying the equation
$$
\eta'' + F = G(\eta, \eta') + F
  $$
by $\eta'$ and use the inequality  $ G(\eta, \eta') + F \geq 0 $ to obtain the result.
\\
$ii)$ \ From Lemma \ref{ineq-G-1} $i)$ we have
$$
\eta'' -  \frac {c_1} {\eta^3} + F \leq 0.
   $$
We multiply  by $\eta'$ to end the proof.
\end{proof}
Let $D_1$ and $D_2$ be defined by
$$
D_1 := \left ( \frac {c_1}{F} \right )^{1/3}
  $$
and
$$
D_2 := 2 \max \left \{ \eta_0, D_1,  \frac 1 F \left ( \frac 1 2
\eta_1^2 + F \eta_0 +
\frac {c_1} {2 \eta_0^2} \right ) ,  \frac 1 F \left (
\frac 1 2 V_2^2 + F D_1 +
\frac {c_1} {2 D_1^2} \right )  \right \}.
  $$
\begin{proposition}
\label{m-2}
$$
\eta(t) < D_2 \quad  \quad \forall t \in [0, T[.
  $$
\end{proposition}
\begin{proof}
By  the contrary we assume that $t_3 > 0 $ is the first time such that
\begin{equation}
\label{et3}
\eta(t_3) = D_2.
\end{equation}
Then, it results that
$ \eta'(t_3) > 0 $ or $ \eta'(t_3) = 0 $ . In the last case,
since $ \eta(t_3) > D_1 $, Lemma \ref{ineq-G-1} $i)$ implies
$$
\eta''(t_3) = G(\eta(t_3), 0) < 0.
  $$
\\
So, in both cases, since $ \eta \in C^2$, there exists
$ t_1 \in [0, t_3[ $ such that
$$
\eta'(t) \geq 0, \quad \forall t \in [t_1, t_3],
  $$
where $t_1$ is the smallest number with this
property.
Two options concerning $t_1$ are possible:
\\
{\bf Option 1:} $ t_1 = 0$.
\ In this case, we have
$$
\eta'(t) \geq 0, \quad \forall t \in [0, t_3].
  $$
From Lemma \ref{l3-1} $ii)$ we obtain
$$
E_2(\eta(t_3), \eta'(t_3)) \leq E_2(\eta_0, \eta_1)
  $$
which implies
$$
\eta(t_3) \leq \frac 1 F \left ( \frac 1 2 \eta_1^2 + F \eta_0 +
\frac {c_1} {2 \eta_0^2}  \right )
  $$
and contradicts \eqref{et3}.
\\
{\bf Option 2:} $ t_1 \in \, ]0, t_3[ $.
\\
We have in this case
$$
\eta'(t_1) = 0, \; \;  \eta'(t) \geq 0 \quad \forall t \in [t_1, t_3]
  $$
which implies
$$
\eta''(t_1) \geq 0 .
  $$
From Lemma \ref{ineq-G-1} $i)$  we obtain
$$
\frac {c_1} {\eta^3(t_1)} \geq F \quad \text { that is }
\quad \eta(t_1) \leq D_1.
   $$
Let $ t_2 \in [t_1, t_3] $ be a time such that
$$
\eta(t_2) = D_1.
   $$
From Lemma \ref{l3-1} $ii)$ and Proposition \ref{m-1}  we have
$$
E_2(\eta(t_3), \eta'(t_3)) \leq E_2(\eta(t_2), \eta'(t_2))
\leq \frac 1 2 V_2^2 + F D_1 +
\frac {c_1} {2 D_1^2}
   $$
which implies
$$
\eta(t_3) \leq \frac 1 F \left ( \frac 1 2 V_2^2 + F D_1 +
\frac {c_1} {2 D_1^2} \right )
   $$
and contradicts \eqref{et3} and the proof ends.
\end{proof}
We define $V_3$ as follows
\begin{equation}
\label{B5}
V_3 := \max \left\{ 1 - \eta_1 , \; 2 \sqrt{2 F D_2}, \;
2 \sqrt{\eta_1^2 + 2 F \eta_0} \right\}.
\end{equation}
\begin{proposition}
\label{m-3}
$$
\eta'(t) > - V_3, \quad \forall t \in [0, T[.
  $$
\end{proposition}
\begin{proof}
We argue by the contrary and assume that $ t_2 \in ]0, T[ $ is the first time such that
\begin{equation}
\label{epb5}
\eta'(t_2) = - V_3.
\end{equation}
We have two options:
\\
{\bf Option I.} $ \eta'(t) \leq 0, \quad \forall t \in [0, t_2] $.
\\
From Lemma \ref{l3-1} $i)$ we have
$$
E_1(\eta(t_2), \eta'(t_2)) \leq E_1(\eta_0, \eta_1)
  $$
which implies
$$
\frac 1 2 | \eta'(t_2) |^2 \leq \frac 1 2 \eta_1^2 + F \eta_0
   $$
and  contradicts \eqref{epb5}.
\\
{\bf Option II:} There exists $ t_1 \in \; ]0, t_2[ $ such that
$$
\eta'(t_1) = 0 \quad \quad \text { and } \quad \eta'(t) \leq 0
\quad \forall t \in [t_1, t_2].
  $$
Then
$$
E_1(\eta(t_2), \eta'(t_2)) \leq E_1(\eta(t_1), 0)
  $$
which combined with  Proposition \ref{m-2} implies
$$
\frac 1 2  | \eta'(t_2) |^2 \leq F D_2
  $$
and  contradicts \eqref{epb5}.
\end{proof}
The most difficult part is to obtain a lower bound of $\eta$ (Proposition \ref{m-4}).
Before we remark that from Corollary \ref{c2} we have
\begin{equation}
\label{G-s1-s2}
G(\beta, \gamma) \geq c_3 \beta^{- s_1} - c_4 \gamma \beta^{-1 - s_2}
- F, \quad
\forall \beta \in ]0, \beta_0], \quad \forall \gamma \leq 0
\end{equation}
where  $ \beta_0, c_3, c_4 $ were defined in Lemma \ref{l3} and
\begin{equation}
\label{s-1}
s_1 =
\begin{cases}
2 \left ( 1 - \frac 1 \alpha \right )  \quad \text{in Case I (line contact)}
\\
2  - \frac 3 \alpha   \quad \quad \; \; \; \text{in Case II (point contact)}
\end{cases}
\end{equation}
\begin{equation}
\label{s-2}
s_2 =
\begin{cases}
2  - \frac 3 \alpha  \quad \text{in Case I}
\\
2  - \frac 4 \alpha   \quad \text{in Case II},
\end{cases}
\end{equation}
Notice that $s_1>0$ and $s_2 \geq 0$.
Let $D_3$ and $D_4>0$ be defined by
\begin{equation}
\label {B6}
D_3 := \frac 2 3 \min \left\{ \eta_0, \; \left (   \frac {c_3} F \right )^{1/s_1} \right\},
\end{equation}
\begin{equation}
\begin{cases}
D_4 = \frac 1 2 \min \{ \eta_0, \;  \beta_0, \;
\left (  \frac {c_3} F \right )^{1/s_1}, \;
\left ( D_3^{- s_2} + \frac {s_2} {c_4} V_3 \right )^{- 1/s_2} \}
\quad \text { if } \; s_2 > 0
\\
\text{ and }
\\
D_4 = \frac 1 2 \min \{ \eta_0, \;  \beta_0, \; \left (  \frac {c_3} F
\right )^{1/s_1}, \;
D_3 e^{-V_3/c_4} \}  \quad \text { if } \; s_2 = 0.
\end{cases}
\end{equation}

\begin{proposition}
\label{m-4}
Under assumption
\begin{equation}
\begin{cases}
\alpha \geq \frac 3 2  \quad \text{in Case I (line contact)}
\\
\text {or}
\\
\alpha \geq 2  \quad \text{in Case II (point contact)}
\end{cases}
\end{equation}
we have $ \eta(t) > D_4, \quad \forall t \in [0, T[ $.
\end{proposition}
\begin {proof}
By the contrary we assume  $t_2  \in \, ]0, T[  $ is the first time such that
\begin{equation}
\label{et2b7}
\eta(t_2) = D_4.
\end{equation}
Notice that $ \eta_0 > D_3 > D_4. $
Let  $t_1 \in ]0, t_2[  $ be the last point where
\begin{equation}
\label{et1b6}
\eta(t_1) = D_3.
\end{equation}
By definition of  $D_3$ we have:
\begin{equation}
\label{2-ineq}
\begin{cases}
\eta'(t_1) \leq 0, \quad   \eta'(t_2) \leq 0
\\
c_3 \eta(t)^{- s_1} > F, \quad \forall t \in [t_1, t_2].
\end{cases}
\end{equation}
We first see
\begin{equation}
\label{etpinf0}
\eta'(t) \leq 0, \quad \forall t \in [t_1, t_2].
\end{equation}
Suppose that \eqref{etpinf0} is false, then there exists $\tau \in ]t_1, t_2[ $ such that
%
$$
\eta'(\tau) > 0.
  $$
%
Let  $\tau_1$ be the supremum of  $\tau \in ]t_1, t_2[ $ satisfying  $\eta'(\tau) > 0.$
It is clear that $ \tau_1 < t_2 $ and it is a local maximum of $\eta$ which implies
\begin{equation}
\label{a1}
\begin{cases}
\eta'(\tau_1) = 0
\\
\eta''(\tau_1) \leq 0.
\end{cases}
\end{equation}
Then  from \eqref{G-s1-s2} and \eqref{2-ineq}
we have
$$
\eta''(\tau_1)  = G(\eta(\tau_1), 0) \geq \frac {c_3} {\eta(\tau_1)^{s_1}} - F > 0
  $$
which  contradicts (\ref{a1}).
Then \eqref{etpinf0} is proved.
\\
Combining \eqref{G-s1-s2} and \eqref{2-ineq} we deduce
\begin{equation}
\label{eppg}
\eta''\geq - c_4 \eta' \eta^{-1 - s_2}  \quad \text{ on } \;  [t_1, t_2].
\end{equation}
{\bf Case $i)$:} \ \ $ s_2 > 0 $.
\\
We integrate \eqref{eppg} to deduce
$$
\eta'(t) \geq \eta'(t_1) + \frac {c_4}{s_2} \eta(t)^{- s_2}  - \frac {c_4}{s_2} \eta(t_1)^{- s_2}, \quad
\forall t \in [t_1, t_2]
  $$
and thanks to  \eqref{etpinf0} and Proposition \ref{m-3} applied
for $t = t_1$ we obtain
$$
\frac {c_4}{s_2} \eta(t_2)^{- s_2} \leq \frac {c_4}{s_2}
\eta(t_1)^{- s_2} + V_3.
  $$
Since $\eta(t_1) = D_3$ it results
$$
\eta(t_2) \geq \left ( D_3^{-s_2} + \frac {s_2} {c_4} V_3 \right )^{- 1/s_2}
  $$
which contradicts \eqref{et2b7}.
\\
{\bf Case $ii)$:}  $ s_2 = 0$.
\\
We integrate \eqref{eppg} to obtain
$$
\eta'(t) \geq \eta'(t_1) +  c_4 \log \left ( \frac 1 {\eta(t)} \right )
-  c_4 \log \left ( \frac 1 {\eta(t_1)} \right ) ,  \quad \forall t \in [t_1, t_2]
  $$
which implies
$$
c_4 \log \left ( \frac 1 {\eta(t_2)} \right )  \leq c_4 \log \left (
  \frac 1 {\eta(t_1)} \right )  + V_3.
  $$
Then
$$
\eta(t_2) \geq D_3 e^{- V_3/c_4}
  $$
which  contradicts \eqref{et2b7}.
\end{proof}
\subsection{Bounds on $\eta$ for the flat case}
We consider the case $h_0 \equiv 0$.
\\
Let us introduce the auxiliary function $w$ defined as the unique solution
to the problem
\begin{equation}
\label{def-w}
\begin{cases}
- \Delta w = 1 \quad & \text{ in } \; \Omega
\\
w = 0  \quad & \text{ on  } \; \partial \Omega
\end{cases}
\end{equation}
and define the constant $C(\Omega)$ by
$$
C(\Omega) = \int_\Omega w(x) \, dx.
  $$
By maximum principle we have $w >0$ on $\Omega$
which implies
$$
C(\Omega) > 0.
  $$
In the following, for any real number $z$ we   denote
$ z^+ = \max\{z, 0\} $ (positive part) and $ z^- = - \min\{z, 0\} $
(negative part).  We have the identity $ z = z^+ - z^- $.
\begin{lemma}
\label{Geq3}
$\eta$ satisfies the following differential equation
$$
\eta'' = C(\Omega) \frac {(\eta')^-} {\eta^3} - F.
  $$
\end{lemma}
\begin{proof}

For $ h_0 \equiv 0 $ the inequality (1.7) becomes
\begin{equation}
\label{ineq-h0-0}
\beta^3 \int_\Omega \nabla q \cdot \nabla (\varphi - q)
\geq - \gamma \int_\Omega (\varphi - q), \quad \forall \, \varphi
\in K.
\end{equation}
The required result is a direct consequence of the following facts
\\
- if $ \gamma \geq 0 $ the solution of \eqref{ineq-h0-0} is $q = 0$
\\
- if $ \gamma < 0 $ the solution of \eqref{ineq-h0-0} is
$ q = - \frac {\gamma w} {\beta^3}$.
\end{proof}
%
%
%
%
%
The bounds on $\eta$ and $\eta'$ can be summarized in the following proposition
\begin{proposition}
\label{ml3}
The following inequalities are valid:
\\
{\bf I)} \ For $\eta_1 \leq 0 $ and $ t \in \, ]0, T[ $ we have
\\
\hspace*{2mm} {\bf Ia)} \quad $ - \sqrt{\eta_1^2 + 2 F \eta_0} \leq
\eta'(t) < 0 $
\\
\hspace*{2mm} {\bf Ib)} \quad $ \eta_0 \left [ \frac {C(\Omega)}
{C(\Omega) + 2 \eta_0^2 F t - 2 \eta_0^2 \eta_1  }
\right ]^{1/2}  \leq \eta(t) \leq \eta_0.$
\\
{\bf II)} \ For any $\eta_1 > 0$ we define $t_0 = \frac {\eta_1} F $
and ${\hat \eta}_0 = \eta_0 + \frac {\eta_1^2} {2 F} $, then
$t_0 < T$ and we have
\\
\hspace*{2mm} {\bf IIa)} \quad $ \eta(t) = - \frac 1 2 F t^2 + \eta_1
t + \eta_0 $ \ for \ $ t \in [0, t_0] $
\\
\hspace*{2mm} {\bf IIb)} \quad $ - \sqrt{2 F {\hat \eta_0}} \leq
\eta'(t) < 0 $ \ for \ $ t \in \, ]t_0, T[ $
\\
\hspace*{2mm} {\bf IIc)} \quad $ {\hat \eta_0} \left [ \frac {C(\Omega)}
{C(\Omega) + 2 {\hat \eta_0}^2 F (t - t_0)  }
\right ]^{1/2}  \leq \eta(t) \leq {\hat \eta_0}$
\ for \ $ t \in \, ]t_0, T[. $
\end{proposition}
\begin{proof}
{\bf I)} \ We assume that $ \eta_1 \leq 0$.
Then, Lemma \ref{Geq3} implies $\eta''(0) = -F$ and therefore
there exists
 a point $t_1 \in \, ]0, T] $ such that
$$
\eta'(t) < 0, \quad \forall \, t \in \, ]0, t_1[,
  $$
where  $t_1$ denotes the largest element with this
property.
We now prove
\begin{equation}
\label{t1eT}
\eta'(t) < 0, \quad \forall \, t \in \,  ]0, T[,
\end{equation}
which is equivalent to assert $ t_1 = T $.
\\
In order to prove (\ref{t1eT}) we argue by contradiction and assume
$ t_1 < T $
and $\eta'(t_1) = 0$ which implies \
$ \eta''(t_1) \geq 0 $ \ and  contradicts \ $ \eta''(t_1) = -F $ \ which is obtained from Lemma
\ref{Geq3} and
proves \eqref{t1eT}.  \eqref{t1eT} implies
\begin{equation}
\label{eqeta-fc}
\eta'' = - C(\Omega) \frac {\eta'} {\eta^3} - F \quad \text{ on }
\; \; [0, T[.
\end{equation}
We multiply \label{eqeta-fc} by $\eta'$ to obtain that
$ \frac 1 2 (\eta')^2 + F \eta $ is a non-increasing function
on  $[0, T]$. This completes the proof of the
double inequality in {\bf Ia)}.
\\
Since $\eta$ is a non-increasing function, we deduce  the inequality of the right-hand side of {\bf Ib)}.
Now we integrate  \eqref{eqeta-fc} over $ [0, t[ $ to obtain
$$
\eta' = \eta_1 + \frac {C(\Omega)} {2 \eta^2}
- \frac {C(\Omega)} {2 \eta_0^2} -  F t.
  $$
Thanks to $\eta' < 0 $ of {\bf Ia)} we obtain
$$
\frac {C(\Omega)} {2 \eta^2} <
\frac {C(\Omega)} {2 \eta_0^2} + F t - \eta_1
\quad \text { on } \; \; ]0, T[
  $$
which completes the proof of {\bf Ib)}.
\\
{\bf II)} we assume that $\eta_1 > 0$. Then, Lemma
\ref{Geq3} implies
$ \eta'' = - F $ for  $t \in [0, t_0]$ which proves
{\bf IIa)}. Since $\eta'(t_0) = 0$ and $\eta(t_0) = {\hat \eta}_0$
the proofs of {\bf IIb)} and {\bf IIc)} are similar to the proofs
of {\bf Ia)} and  {\bf Ib)}  respectively.

\end{proof}
\subsection{Proofs of the theorems}
{\it Proof of Theorem \ref{G-lipsch}}.
Let us fix $ \beta > 0 $ and $ \gamma \in \R $ and take
$\tilde{\beta} > 0$ and $\tilde{\gamma} \in \R $ such that
$ (\tilde{\beta},\tilde{\gamma}  ) $ are close enough to $(\beta,
\gamma )$.
We denote by $\tilde{q} \in K$ the solution to the Reynolds inequality
\begin{equation}
\nonumber
\displaystyle
\int_\Omega\left (h_0 + \tilde{\beta} \right )^3 \nabla
\tilde{q} \cdot \nabla
(\varphi - \tilde{q}) \geq \int_\Omega h_0
\dfrac{\partial }{\partial x_1} (\varphi-\tilde{q}) - \tilde{\gamma}
\int_\Omega (\varphi-\tilde{q}), \quad \forall \varphi \in K
\end{equation}
which can be written in the form
\begin{equation}
\displaystyle
\begin{split}
\int_\Omega\left (h_0 + \beta \right )^3 \nabla
\tilde{q} \cdot \nabla
(\varphi - \tilde{q})  & \geq \int_\Omega h_0
\dfrac{\partial }{\partial x_1} (\varphi-\tilde{q}) - \tilde{\gamma}
\int_\Omega (\varphi-\tilde{q}) +
\\
& + (\beta - \tilde{\beta})
\int_\Omega A_{\beta, \tilde{\beta}} (x) \nabla \tilde{q} \cdot
\nabla ( \varphi - \tilde{q} )
\end{split}
\label{7-pr}
\end{equation}
where $ A_{\beta, \tilde{\beta}} $ is uniformly bounded in
$ \tilde{\beta} $.
\\
We take $ \varphi = \tilde{q} $ in \eqref{7}, $ \varphi = q $ in
\eqref{7-pr} and we add both inequalities to get
\begin{equation}
\nonumber
\displaystyle
\begin{split}
\int_\Omega\left (h_0 + \beta \right )^3 | \nabla ( \tilde{q} - q )
|^2 & \leq | \tilde{\gamma} - \gamma | | \Omega |^{1/2}
\| \tilde{q} - q \|_{L^2(\Omega)} +
\\
& + | \tilde{\beta} - \beta |
\| A_{\beta, \tilde{\beta}} \|_{L^\infty(\Omega)}
\| \nabla \tilde{q} \|_{L^2(\Omega)} \| \nabla
(\tilde{q} - q) \|_{L^2(\Omega)}.
\end{split}
\end{equation}
By  Poincar\'e inequality  the proof ends.
\\
{\it Proof of Theorem \ref{th-ex-station}}.
\\
Let us introduce the function $ g : ]0, + \infty [ \rightarrow \R $ defined by
$$
g(\beta) = G(\beta, 0).
   $$
From Theorem \ref{t1}, it is clear that $g$ is continuous.
Lemma \ref{ineq-G-1} $i)$ and Corollary \ref{c2} imply
$$
\lim_{\beta \rightarrow + \infty} g(\beta) = - F \quad \text{ and }
\quad  \lim_{\beta \rightarrow 0} g(\beta) = + \infty
   $$
which proves the theorem.
\\
{\it Proof of Theorem \ref{th-ev}}.
\\
The result is a consequence of Propositions
\ref{m-1}, \ref{m-2}, \ref{m-3} and \ref{m-4}.
\\
{\it Proof of Theorem \ref{th-ev-fc}}.
\\
The non-existence of stationary solutions comes  from Lemma \ref{Geq3}.
The results concerning the evolution on the position follow from  Proposition \ref{ml3}.
%
%

\end{document}